\documentclass[12pt,reqno]{amsart}

\usepackage[T1]{fontenc}
\usepackage[utf8]{inputenc}
\usepackage{amsmath,amsfonts,amsthm,amssymb,amsxtra}
\usepackage{bbm} 
\usepackage[parfill]{parskip}    
\usepackage{graphicx}

\usepackage[ngerman, german, english]{babel} 

\newtheorem{theorem}{Theorem}
\newtheorem{proposition}[theorem]{Proposition}
\newtheorem{lemma}[theorem]{Lemma}

\theoremstyle{definition}

\newtheorem{conjecture}[theorem]{Conjecture}

\theoremstyle{remark}

\numberwithin{equation}{section}
\numberwithin{figure}{section}

\renewcommand{\epsilon}{\varepsilon}

\newcommand{\loc}{{\rm loc}}

\newcommand{\N}{\mathbb{N}}

\newcommand{\R}{\mathbb{R}}

\newcommand{\eps}{\varepsilon}

\newcommand{\sph}{{\mathbb{S}}}

\DeclareMathOperator{\dist}{dist}

\begin{document}

\title[An exceptional property of the 1D Bianchi-Egnell inequality]{An exceptional property of the one-dimensional Bianchi-Egnell inequality}

\author[T. König]{Tobias König}
\address[T. König]{Goethe-Universität Frankfurt, Institut für Mathematik, 
Robert-Mayer-Str. 10, 60325 Frankfurt am Main, Germany.}
\email{koenig@mathematik.uni-frankfurt.de}

\begin{abstract}
In this paper, for $d \geq 1$ and $s \in (0,\frac{d}{2})$, we study the Bianchi-Egnell quotient
\[ \mathcal Q(f) = \inf_{f \in \dot{H}^s(\mathbb R^d) \setminus \mathcal B} \frac{\|(-\Delta)^{s/2} f\|_{L^2(\mathbb R^d)}^2 - S_{d,s} \|f\|_{L^{\frac{2d}{d-2s}}(\mathbb R^d)}^2}{\text{dist}_{\dot{H}^s(\mathbb R^d)}(f, \mathcal B)^2}, \qquad f \in \dot{H}^s(\mathbb R^d) \setminus \mathcal B, \]
where $S_{d,s}$ is the best Sobolev constant and $\mathcal B$ is the manifold of Sobolev optimizers. By a fine asymptotic analysis, we prove that when $d = 1$, there is a neighborhood of $\mathcal B$ on which the quotient $\mathcal Q(f)$ is larger than the lowest value attainable by sequences converging to $\mathcal B$. This behavior is surprising because it is contrary to the situation in dimension $d \geq 2$ described recently in \cite{Koenig}. 

This leads us to conjecture that for $d = 1$, $\mathcal Q(f)$ has no minimizer on $\dot{H}^s(\mathbb R^d) \setminus \mathcal B$, which again would be contrary to the situation in $d \geq 2$. 

As a complement of the above, we study a family of test functions which interpolates between one and two Talenti bubbles, for every $d \geq 1$. For $d \geq 2$, this family yields an alternative proof of the main result of \cite{Koenig}. For $d =1$ we make some numerical observations which support the conjecture stated above. 
 \end{abstract}

\date{February 29, 2024}
\thanks{\copyright\, 2024 by the author. This paper may be reproduced, in its entirety, for non-commercial purposes. The author thanks Tobias Weth for helpful discussions. \\ \emph{Data availability:} Data sharing not applicable to this article as no datasets were generated or analysed during
the current study.}

\maketitle

\section{Introduction and main result}

For any space dimension $d \geq 1$, the Sobolev inequality on $\R^d$ for $s \in (0,\frac{d}{2})$ reads as 
\begin{equation}
\label{sobolev Rd}
\|(-\Delta)^{s/2} f\|_{L^2(\R^d)}^2 \geq S_{d,s} \|f\|_{L^{p}(\R^d)}^2, \qquad f \in \dot{H}^s(\R^d),
\end{equation}
where 
\[p = \frac{2d}{d - 2s}.\]

 The unique optimizers of \eqref{sobolev Rd} are the \emph{bubble functions} \cite{Au, Ta, Li} 
\begin{equation}
\label{M definition}
\mathcal B := \left \{ x \mapsto c (a + |x-b|^2)^{-\frac{d-2s}{2}} \, : \, a > 0, \, b \in \R^d, \, c \in \R \setminus \{0\} \right \}. 
\end{equation}
It turns out that the \emph{quantitative stability} of \eqref{sobolev Rd} can be formulated in terms of the quotient
\[ \mathcal Q(f) =  \frac{\|(-\Delta)^{s/2} f\|_{L^2(\mathbb R^d)}^2 - S_{d,s} \|f\|_{L^{p}(\mathbb R^d)}^2}{\dist_{\dot{H}^s(\mathbb R^d)}(f, \mathcal B)^2}, \qquad f \in \dot{H}^s(\R^d) \setminus \mathcal B, \]
between the 'Sobolev deficit' and the $\dot{H}^s(\mathbb R^d)$-distance of the optimizers. 
A famous inequality due, for $s = 1$, to Bianchi and Egnell \cite{BiEg} (see \cite{ChFrWe} for general $s \in (0, \frac{d}{2})$) says that there is $c_{BE}(s) > 0$ such that 
\begin{equation}
\label{bianchi egnell Rd}
\mathcal Q(f) \geq c_{BE}(s) \qquad \text{ for all } f \in \dot{H}^s(\R^d). 
\end{equation}
This inequality is optimal in the sense that in the denominator $\dist_{\dot{H}^s(\R^d)}(f, \mathcal B)^2$ of $\mathcal Q(f)$, the distance cannot be measured by a stronger norm than $\dot{H}^s(\R^d)$ and the exponent $2$  cannot be replaced by a smaller one if inequality \eqref{bianchi egnell Rd} is to be satisfied with a strictly positive constant on the right hand side. 

The proof strategy of Bianchi and Egnell (and of \cite{ChFrWe} for general $s \in (0, d/2)$) consists in proving first that 
\[ \liminf_{n \to \infty} \mathcal Q(f_n) \geq c_{BE}^{\text{loc}}(s) \]
for every sequence $f_n$ which converges (in $\dot{H}^s(\R^d)$) to $\mathcal B$, for an explicit constant $c_{BE}^{\text{loc}}(s) = \frac{4s}{d + 2s + 2}$. Then \eqref{bianchi egnell Rd} can be deduced from this by a compactness argument. For $s = 1$, using completely new ideas, the authors of the paper \cite{DoEsFiFrLo} have recently given a quantitative version of this argument  which permits to give an explicit lower bound on $c_{BE}(1)$.

By definition, we have $c_{BE}^\text{loc}(s) \geq c_{BE}(s)$. Only recently, it has been observed in \cite{Koenig} that when $d \geq 2$, there is $\rho \in \dot{H}^s(\R^d)$ such that for every $\eps > 0$ small enough, the function $f_\eps(x) = (1 + |x|^2)^{-\frac{d-2s}{2}} + \eps \rho(x)$ satisfies
\[ \mathcal Q(f_\eps) < c_{BE}^{\text{loc}}(s). \]
In particular, this shows that in fact $c_{BE}(s) < c_{BE}^{\text{loc}}(s)$ strictly, for every $d \geq 2$ and $s \in (0, \frac{d}{2})$. 

Therefore, the following theorem about the situation in dimension $d = 1$, which \cite{Koenig} makes no statement about, comes as a surprise. 

\begin{theorem}
\label{theorem d=1}
Let $d = 1$ and $s \in (0,\frac{1}{2})$. Then there is a neighborhood $U \subset \dot{H}^s(\R)$ of $\mathcal B$ such that for every $f \in U \setminus \mathcal B$ one has 
\[ \mathcal Q(f) \geq c_{BE}^{\text{loc}}(s). \]
\end{theorem}

In view of Theorem \ref{theorem d=1} and the preceding discussion it is tempting to formulate the following conjecture: 

\begin{conjecture}
\label{conjecture BE}
Let $d = 1$ and $s \in (0, \frac{1}{2})$. Then $c_{BE}(s) = c_{BE}^{\text{loc}}(s)$ and any minimizing sequence for $c_{BE}(s)$ converges to $\mathcal B$. In particular, \eqref{bianchi egnell Rd} does not admit a minimizer. 
\end{conjecture}

We provide some more evidence for Conjecture \ref{conjecture BE} in Section \ref{section conjecture}.

Let us stress here only that also the last part of the conjecture about non-existence of a minimizer would be in contrast to the situation in higher dimensions. Indeed, for $d \geq 2$, the recent result from \cite{Koenig-min} shows that for any $s \in (0, \frac{d}{2})$, every minimizing sequence for $c_{BE}(s)$ converges (up to conformal symmetries and taking subsequences) to some minimizer $f \in \dot{H}^s(\R^d) \setminus \mathcal B$. A key ingredient in the proof in \cite{Koenig-min} is the strict inequality $c_{BE}(s) < c_{BE}^{\text{loc}}(s)$, which fails in $d= 1$ if the first part of Conjecture \ref{conjecture BE} is true.

The special role of dimension $d = 1$ for the Bianchi-Egnell inequality apparent from comparing Theorem \ref{theorem d=1} with the results of \cite{Koenig} is somewhat reminiscent of the situation in \cite[Theorem 2]{FrKoTa}. There, similarly to \cite{Koenig}, for a family of so-called reverse Sobolev inequalities of order $s > d/2$ a certain behavior of the inequality can be verified for $d \geq 2$ through an appropriate choice of test functions, but the same test functions do not yield the result if $d = 1$. To our knowledge, complementing \cite[Theorem 2]{FrKoTa} for $d = 1$ is still an open question. We are currently not in a position to convincingly explain the origin of the particular behavior of $d=1$ in either of the settings studied in this paper and in \cite{FrKoTa}. It would therefore be very interesting to shed some further light on the role of $d = 1$ in conformally invariant minimization problems of fractional order.

The stability of Sobolev's and related functional inequalities and the fine properties of the minimization problem \eqref{bianchi egnell Rd} and its analogues is currently a very active topic of research with many recent contributions. Without attempting to be exhaustive, we mention in particular that the methods and results from \cite{Koenig} and \cite{Koenig-min} have been recently extended to the Log-Sobolev inequality \cite{ChLuTa} and to the Caffarelli--Kohn--Nirenberg inequality in the preprints \cite{WeWu} and \cite{DeTi}; see also \cite{FrPe}. Besides, an excellent introduction to the topic is provided by the recent lecture notes \cite{Frank2023b}.

\section{The Bianchi-Egnell inequality on $\mathbb S^d$. } 
\label{section sphere}

The inequalities \eqref{sobolev Rd} and \eqref{bianchi egnell Rd} have a conformally equivalent formulation on the $d$-dimensional sphere $\mathbb S^d$ (viewed as a subset of $\R^{d+1}$), which is in some sense even more natural than that on $\R^d$ and, at any rate, more convenient for the arguments used in this paper. 

The conformal map between $\R^d$ and $\mathbb S^d$ which induces this equivalence is the (inverse) \emph{stereographic projection} $\mathcal S: \R^d \to \mathbb S^d$ given by 
\begin{equation}
\label{ster proj definition}
(\mathcal S(x))_i = \frac{2 x_i}{1 + |x|^2} \quad (i = 1,...,d), \qquad (\mathcal S(x))_{d+1} = \frac{1 - |x|^2}{1+|x|^2}. 
\end{equation} 
We denote by $J_{\mathcal S}(x)= |\det D \mathcal S(x)| = \left(\frac{2}{1 + |x|^2}\right)^d$ its Jacobian determinant. If $f \in \dot{H}^s(\R^d)$ and $u \in H^s(\mathbb S^d)$ are related by 
\begin{equation}
\label{conf trafo}
f(x) = u_{\mathcal S}(x) := u(\mathcal S(x)) J_\mathcal S(x)^{1/p}, 
\end{equation} 
then the Sobolev inequality \eqref{sobolev Rd} translates to 
\begin{equation}
\label{sobolev}
(u, P_s u) \geq S_{d,s} \|u\|_{L^p(\mathbb S^d)}^2 \qquad \text{ for } u \in H^s(\sph^d), 
\end{equation} 
where $(\cdot, \cdot)$ is $L^2(\sph^d)$ scalar product. The operator $P_s$ appearing here is given on spherical harmonics $Y_\ell$ of degree $\ell \geq 0$ by 
\[ P_s Y_\ell = \alpha(\ell) Y_\ell \]
with 
\[ \alpha(\ell) = \frac{\Gamma(\ell + \frac{d}{2} + s)}{\Gamma(\ell + \frac{d}{2} - s)}. \]

The manifold of optimizers of \eqref{sobolev} (i.e. the image of the bubble functions $\mathcal B$ under the transformation \eqref{conf trafo})  is given by 
\begin{equation}
\label{M definition}
\mathcal M = \left\{ \omega \mapsto c (1 + \zeta \cdot \omega)^{-\frac{d-2s}{2}} \, : \, c \in \R \setminus \{0\}, \, \zeta \in \R^{d+1}, \, |\zeta| < 1 \right\} .   
\end{equation}

We refer to \cite{Frank2023b} for a justification of these facts. 

Finally, the Bianchi-Egnell quotient on $\mathbb S^d$ reads
\[ \mathcal E(u) := \frac{(u, P_s u) - S_{d,s} \|u\|_{L^p(\mathbb S^d)}^2 }{\inf_{h \in \mathcal M} (u - h, P_s (u-h))}, 
\]
and the stability inequality corresponding to \eqref{bianchi egnell Rd} is
\begin{equation}
\label{bianchi egnell}
\mathcal E(u) \geq c_{BE}(s) \qquad \text{ for all } u \in H^s(\sph^d). 
\end{equation}
(Notice carefully that the constants $S_{d,s}$ and $c_{BE}(s)$ do not change as one passes from $\R^d$ to $\mathbb S^d$.)

\textit{Notation and conventions.  } We always consider $H^s(\sph^d)$ to be equipped with the norm $\|u\| := (u, P_s u)^{1/2}$, which is equivalent to the standard norm on $H^s(\sph^d)$. Here, $(\cdot, \cdot)$ is $L^2(\sph^d)$ scalar product.

We always assume that the numbers $p$, $d$, and $s$ satisfy the relation $s \in (0, \frac{d}{2})$ and $p = \frac{2d}{d-2s}$. 

For any $q \in [1, \infty]$, we denote for short $\|\cdot\|_q := \|\cdot\|_{L^q(\mathbb S^d)}$. Also, for brevity we will often write the integral of a real-valued function $u$ defined on $\mathbb S^d$ as $\int_{\mathbb S^d} u$ instead of $\int_{\mathbb S^d} u(\omega) \, d \omega$. The implied measure $d \omega$ is always taken to be non-normalized standard surface measure, so that $\int_{\mathbb S^d} 1 \, d \omega = |\mathbb S^d|$. 

For $\ell \geq 0$, we denote by $E_\ell$ (resp. $E_{\geq \ell}$ or $E_{\leq \ell}$) the space of spherical harmonics of $\mathbb S^d$ of degree equal to (resp. at least or at most) $\ell$. 

\section{Proof of Theorem \ref{theorem d=1}}

In this section, our goal is to prove the following. 

\begin{theorem}
\label{theorem d=1 Sd}
Let $d = 1$ and $s \in (0,\frac{1}{2})$. Then there is a neighborhood $V \subset H^s(\sph^1)$ of $\mathcal M$ such that for every $u \in V \setminus \mathcal M$ one has 
\[ \mathcal E(u) \geq c_{BE}^{\text{loc}}(s). \]
\end{theorem}

The lower bound $c_{BE}^{\text{loc}}(s)$ in Theorem \ref{theorem d=1 Sd} is sharp. Indeed, for $u_\mu(\theta) := 1 + \mu \sin(2 \theta)$ the computations in the proof below (or their less involved variant in \cite[proof of Theorem 1]{Koenig}) show that $\text{dist}(u_\mu, \mathcal M) \to 0$ and $\mathcal E(u_\mu) \to c_{BE}^{\text{loc}}(s)$ as $\mu \to 0$.

By the equivalence between $\R^d$ and $\mathbb S^d$ explained in Section \ref{section sphere}, Theorem \ref{theorem d=1 Sd} is equivalent to Theorem \ref{theorem d=1} via stereographic projection. For our arguments, it is however much more convenient to work in the setting of the sphere for at least two reasons. Firstly,  Taylor expansions near $\mathcal M$ are simpler because we can choose the constant function $1 \in \mathcal M$ as a basepoint. Secondly, we have the explicit eigenvalues $\alpha(\ell)$ of the operator $P_s$ and their associated eigenfunctions $\sin(\ell \theta)$ and $\cos(\ell \theta)$ at our disposition. 

We now begin with giving the proof of Theorem \ref{theorem d=1 Sd}. Our proof is inspired by \cite{Frank2023}, where in a similar but different situation a certain term in a Taylor expansion vanishes, but the sign of a certain next-order correction term can be recovered. 

Let $(u_n)$ be an arbitrary sequence such that $\text{dist}_{H^s(\sph^1)}(u_n, \mathcal M) \to 0$. 
Theorem \ref{theorem d=1} follows if we can prove that $\mathcal E(u_n) \geq c_{BE}^{\text{loc}}(s)$ for every $n$ large enough. We may thus assume 
\begin{equation}
\label{E(u_n) to cloc}
\mathcal E(u_n) \to c_{BE}^{\text{loc}}(s),
\end{equation}
for otherwise the inequality $\mathcal E(u_n) \geq c_{BE}^{\text{loc}}(s)$ is automatic for large enough $n$.

By \cite[Lemma 3.4]{DeNiKo} (respectively the conformally equivalent statement on $\sph^1$ instead of $\R$), $\text{dist}_{H^s(\sph^1)}(u_n, \mathcal M)$ is achieved by some $h_n \in \mathcal M$.

Up to multiplying $u_n$ by a constant $c_n$ and applying a conformal transformation $T_n$, both of which do not change $\mathcal E(u_n)$, we may assume that $h_n = 1$. 

The fact that $1 \in \mathcal M$ minimizes $\text{dist}_{H^s(\sph^1)}(u_n, \mathcal M)$ implies that $u_n - 1$ is orthogonal (in $H^s(\mathbb S^1)$) to the tangent space $T_1 \mathcal M$. But this tangent space is precisely spanned by the functions $1$, $\omega_1$ and $\omega_2$. Thus we have $\rho_n := u_n - 1 \in E_{\geq 2}$. We can thus summarize our chosen normalization as 
\begin{equation}
\label{u = 1 + rho}
u_n = 1 + \rho_n, \qquad \rho_n \in E_{\geq 2}, \quad \rho_n \to 0 \, \text{ in } \, H^s(\sph^1). 
\end{equation}

The next lemma refines the decomposition \eqref{u = 1 + rho} and gives additional information. 

\begin{lemma}
\label{lemma r + eta}
Let $(u_n)$ satisfy \eqref{E(u_n) to cloc} and \eqref{u = 1 + rho}. Then there are sequences $\theta_n \in [0, 2 \pi)$, $\mu_n > 0$ and $\eta_n \in E_{\geq 3}$ such that 
\[ u_n = 1 + \mu_n(r_n + \eta_n), \]
with $r_n(\theta) = \sin(2 (\theta - \theta_n)) \in E_2$, $\mu_n \to 0$ and $\|\eta_n\| \to 0$. 
\end{lemma}

\begin{proof}
By \eqref{u = 1 + rho}, we can write $u_n = 1+ \rho_n =  1 + \tilde{r}_n + \tilde{\eta}_n$ with $\tilde{r}_n \in E_2$ and $\tilde{\eta}_n \in E_{\geq 3}$. Consequently,  
\[ (u_n, P_s u_n) = (1, P_s 1) + (\rho_n, P_s \rho_n) =  (1, P_s 1) + (\tilde{r}_n, P_s \tilde{r}_n) + (\tilde{\eta}_n, P_s \tilde{\eta}_n)  \]
and 
\[ \int_{\sph^1} |u_n|^p = \int_{\sph^1} 1^p + \frac{p(p-1)}{2} \int_{\sph^1} \rho_n^2 + \mathcal O\left(\int_{\sph^1} |\rho_n|^{\min\{3, p\}} + |\rho_n|^p \right).   \]
Using the second-order Taylor expansion $(a + h)^\frac{2}{p} = a^\frac{2}{p} + \frac{2}{p} a^{\frac{2}{p}-1} h  - \frac{p-2}{p^2} a^{\frac{2}{p} - 2} h^2 + o(h^2)$, we obtain
\[ \|u_n\|_p^2 = (2 \pi)^\frac{2}{p} + (p-1) (2\pi)^{\frac{2}{p}-1} \int_{\sph^1} \rho_n^2 + o(\|\rho_n\|_2^2). \]
(To produce this form of the error term, we use Hölder's and Sobolev's inequalities together with the facts that $\|\rho_n\| \to 0$ and $2 < \min\{3,p\}$.) 

By \cite[Lemma 12]{Frank2023b}, $\dist_{H^s(\mathbb S^1)}(u_n, \mathcal M)^2 = (\rho_n, P_s \rho_n)$ for $n$ large enough. Using that $S_{1,s}  (p-1) (2\pi)^{\frac{2}{p}-1}  = \alpha(1)$, together with assumption \eqref{E(u_n) to cloc}, we obtain
\begin{align*}
c_{BE}^{\text{loc}}(s) + o(1) &\geq \mathcal E(u_n) = \frac{(\rho_n, P_s \rho_n) - \alpha(1) \int_{\sph^1} \rho_n^2 + o(\|\rho_n\|_2^2)}{(\rho_n, P_s \rho_n)} \\
&=   \frac{(\tilde{r}_n, P_s \tilde{r}_n) - \alpha(1) \int_{\sph^1} \tilde{r}_n^2 + (\tilde{\eta}_n, P_s \tilde{\eta}_n) - \alpha(1) \int_{\sph^1} \tilde{\eta}_n^2}{(\tilde{r}_n, P_s \tilde{r}_n) + (\tilde{\eta}_n, P_s \tilde{\eta}_n)} + o(1).
\end{align*}
Now divide both the numerator and the denominator by $(\tilde{r}_n, P_s \tilde{r}_n)$ and abbreviate $\tau_n := \frac{(\tilde{\eta}_n, P_s \tilde{\eta}_n)}{(\tilde{r}_n, P_s \tilde{r}_n)}$. Since $\tilde{r}_n  \in E_2$, we have 
\[ 1 - \alpha(1) \frac{\int_{\sph^1} \tilde{r}_n^2 }{(\tilde{r}_n, P_s \tilde{r}_n)} = 1 - \frac{\alpha(1)}{\alpha(2)} = c_{BE}^{\text{loc}}(s). \]
Moreover, since $\tilde{\eta}_n \in E_{\geq 3}$, we have $\int_{\sph^1} \tilde{\eta}_n^2 \leq \frac{1}{\alpha(3)} (\tilde{\eta}_n, P_s \tilde{\eta}_n)$. Altogether, the above yields 
\begin{equation}
\label{tau n ineq}
c_{BE}^{\text{loc}}(s) + o(1) \geq \frac{c_{BE}^{\text{loc}}(s) + (1 - \frac{\alpha(1)}{\alpha(3)}) \tau_n}{1 + \tau_n} + o(1). 
\end{equation} 
Since $1 - \frac{\alpha(1)}{\alpha(3)} > 1 - \frac{\alpha(1)}{\alpha(2)} = c_{BE}^{\text{loc}}(s)$, the function $\tau \mapsto \frac{c_{BE}^{\text{loc}}(s) + (1 - \frac{\alpha(1)}{\alpha(3)}) \tau}{1 + \tau}$ is strictly increasing in $\tau \in (0, \infty)$ and takes the value $c_{BE}^{\text{loc}}(s)$ in $\tau = 0$. Together with \eqref{tau n ineq}, this forces $\tau_n = o(1)$. 

Now setting $r_n = \mu_n^{-1} \tilde{r}_n$ and $\eta_n = \mu_n^{-1} \tilde{\eta}_n$ with $\mu_n = \left( \frac{\int_{\sph^1} \tilde{r}_n^2}{\pi}\right)^\frac{1}{2}$ gives the conclusion, noting that $r_n \in E_2$ can be written as $r_n = \sin(2 (\theta- \theta_n))$ for some appropriate $\theta_n \in [0, 2\pi)$. 
\end{proof}

Let us now make the additional assumption that 
\begin{equation}
\label{p geq 4}
p = \frac{2d}{d-2s} \geq 4, 
\end{equation}
which is the case if $s \in [\frac{1}{4}, \frac{1}{2})$. 

Assuming \eqref{p geq 4}, we can prove our main result in a relatively straightforward way. We will explain afterwards how the proof needs to be modified in order to cover the general case.  

\begin{proof}
[Proof of Theorem \ref{theorem d=1 Sd}; easy case $p \geq 4$]

Under the additional assumption \eqref{p geq 4} we are now ready to perform the final expansion needed for the proof of Theorem \ref{theorem d=1}. Thanks to $p \geq 4$, the function $t \mapsto |t|^{p}$ is four times continuously differentiable on $\R$, and thus we may expand, for every value of $\rho_n$, 
\begin{align*}
|1 + \rho_n|^p &= 1 + p \rho_n + \frac{p(p-1)}{2} \rho_n^2 + \frac{p(p-1)(p-2)}{6} \rho_n^3 + \frac{p(p-1)(p-2)(p-3)}{24} \rho_n^4 \\
& \qquad + \mathcal O(|\rho_n|^{\min\{5, p\} }+ |\rho_n|^p) . 
\end{align*}
Let us write $\rho_n = \mu_n (r_n + \eta_n)$ as in Lemma \ref{lemma r + eta}. Since $\int_{\sph^1} \rho_n = 0$, this implies 
\begin{align}
\int_{\sph^1} | 1 + \rho_n|^p  &= \int_{\sph^1} 1^p + \frac{p(p-1)}{2} \mu_n^2 \left( \int_{\sph^1} r_n^2 +  \int_{\sph^1} \eta_n^2 \right) + \frac{p(p-1)(p-2)}{6} \mu_n^3  \left(  \int_{\sph^1} (r_n + \eta_n)^3 \right)  \nonumber \\
&\qquad + \frac{p(p-1)(p-2)(p-3)}{24}  \mu_n^4 \int_{\sph^1} r_n^4 + o(\mu_n^4 ).  \label{proof basic expansion}
\end{align}
Here, the error term comes from estimating, using Lemma \ref{lemma r + eta},
\[ \int_{\sph^1} |\rho_n|^p \lesssim \mu_n^p \int_{\sph^1} r_n^p + \mu_n^p \int_{\sph^1} \eta_n^p \lesssim \mu_n^p = o(\mu_n^4), \]
because $p > 4$ (if $p=4$, the expansion \eqref{proof basic expansion} is exact), and 
\[ \mu_n^4 \int_{\sph^1} r_n^3 \eta_n \lesssim \mu_n^4 \|\eta_n\| = o(\mu_n^4), \]
and alike for $\mu_n^4 \int_{\sph^1} r_n^2 \eta_n^2$,  $\mu_n^4 \int_{\sph^1} r_n \eta_n^3$ and $\mu_n^4 \int_{\sph^1} \eta_n^4$. 
 Similarly, if additionally $p > 5$, 
\[  \int_{\sph^1} |\rho_n|^5 \lesssim \mu_n^5 \int_{\sph^1} r_n^5 + \mu_n^5 \int_{\sph^1} \eta_n^5 = o(\mu_n^4) \]
because $\int_{\sph^1} r_n^5 \lesssim \|r_n\|^5 = \mathcal O(1)$ and  $\int_{\sph^1} \eta_n^5 \lesssim \|\eta_n\|^5 = o(1)$ by Hölder's and Sobolev's inequalities. Thus \eqref{proof basic expansion} is proved.

Next, we note that from $r_n(\theta) = \sin (2 (\theta - \theta_n))$ we necessarily obtain $\int_{\sph^1} r_n^3 = 0$. We emphasize that this property is what makes dimension $d = 1$ special because for $d \geq 2$ there are spherical harmonics $\rho \in E_2$ such that $\int_{\mathbb S^d} \rho^3 \neq 0$, see \cite{Koenig} or Proposition \ref{proposition local two-bubbles} below. 

Again by the second-order Taylor expansion $(a + h)^\frac{2}{p} = a^\frac{2}{p} + \frac{2}{p} a^{\frac{2}{p}-1} h  - \frac{p-2}{p^2} a^{\frac{2}{p} - 2} h^2 + o(h^2)$, we obtain from \eqref{proof basic expansion} that
\begin{align*}
\left(  \int_{\sph^1} |1 + \rho_n|^p \right)^\frac{2}{p} &= (2 \pi)^\frac{2}{p} + (p-1) (2 \pi)^{\frac{2}{p} - 1} \Big(   \mu_n^2 \int_{\sph^1} r_n^2 + \mu_n^2 \int_{\sph^1} \eta_n ^2 \\
& \qquad + \frac{p-2}{3} \mu_n^3 \int_{\sph^1} (3 r_n^2 \eta_n + 3 r_n \eta_n^2 + \eta_n^3) + \frac{(p-2)(p-3)}{12} \mu_n^4   \int_{\sph^1} r_n^4 \Big) \\
& \qquad - \frac{(p-1)^2 (p-2)}{4} (2 \pi)^{\frac{2}{p}-2} \mu_n^4 \left(  \int_{\sph^1} r_n^2 \right)^2 + o(\mu_n^4) \\
&= (2 \pi)^\frac{2}{p} + (p-1) (2 \pi)^{\frac{2}{p} - 1} \mu_n^2 \left( \int_{\sph^1} r_n^2 + \int_{\sph^1}  \eta_n ^2 + (p-2) \mu_n \int_{\sph^1} r_n^2 \eta_n \right) \\
& \qquad + \frac{(p-1)(p-2)}{12} (2 \pi)^{\frac{2}{p} - 1} \mu_n^4 \left( (p-3)   \int_{\sph^1} r_n^4 - \frac{3(p-1)}{2 \pi} \left(  \int_{\sph^1} r_n^2 \right)^2   \right) \\
& \qquad  + o(\mu_n^4 + \mu_n^2 \|\eta_n\|^2). 
\end{align*}
The other terms appearing in $\mathcal E(u_n)$ can be expanded as 
\[ (u_n, P_s u_n) = (1, P_s 1) + \mu_n^2 (r_n, P_s r_n) + \mu_n^2 (\eta_n, P_s \eta_n) \]
and (for $n$ large enough)
\[ \dist(u_n, \mathcal M)^{-2}= (\rho_n, P_s \rho_n)^{-1} = \frac{1}{\mu_n^2 (r_n, P_s r_n)} \left( 1 - \frac{(\eta_n, P_s \eta_n)}{(r_n, P_s r_n)} + o(\|\eta_n\|^2) \right). \] 
We can now put all of these expansions together to expand $\mathcal E(u_n)$. Noticing that $S_{1,s} (p-1) (2 \pi)^{\frac{2}{p} - 1} = \alpha(1)$, we arrive at
\begin{align*}
\mathcal E(u_n) &= \frac{(u_n, P_s u_n) - S_{1,s} \|u_n\|_p^2}{\dist(u_n, \mathcal M)^2} \\
&= \frac{(r_n, P_s r_n) - \alpha(1) \int_{\sph^1} r_n^2}{(r_n, P_s r_n)} \left(1 -  \frac{(\eta_n, P_s \eta_n)}{(r_n, P_s r_n)}\right) \\
&\qquad + \frac{(\eta_n, P_s \eta_n) - \alpha(1) \left(\int_{\sph^1} \eta_n^2 - (p-2) \mu_n \int_{\sph^1} r_n^2 \eta_n \right) }{(r_n, P_s r_n)} + o(\|\eta_n\|^2) \\
& \qquad + \frac{\alpha(1)}{(r_n, P_s r_n)} \frac{p-2}{12} \mu_n^2  \left( \frac{3(p-1)}{2 \pi} \left(  \int_{\sph^1} r_n^2 \right)^2   - (p-3)   \int_{\sph^1} r_n^4   \right) + o(\mu_n^2). 
\end{align*}
Since $r_n \in E_2$, we have $(r_n, P_s r_n) = \alpha(2) \int_{\sph^1} r_n^2$ and consequently 
\[ \frac{(r_n, P_s r_n) - \alpha(1) \int_{\sph^1} r_n^2}{(r_n, P_s r_n)} = 1 - \frac{\alpha(1)}{\alpha(2)} = c_{BE}^{\text{loc}}(s). \]
Therefore we can write the above expansion as 
\begin{align}
& \qquad (r_n, P_s r_n) \left( \mathcal E(u_n) - c_{BE}^{\text{loc}}(s) \right) \nonumber \\
 &= (1 - c_{BE}^{\text{loc}}(s)  + o(1)) (\eta_n, P_s \eta_n) - \alpha(1) \left(\int_{\sph^1} \eta_n^2 - (p-2) \mu_n \int_{\sph^1} r_n^2 \eta_n \right) \nonumber \\
&\quad  + \alpha(1) \frac{p-2}{12} \mu_n^2  \left( \frac{3(p-1)}{2 \pi} \left(  \int_{\sph^1} r_n^2 \right)^2   - (p-3)   \int_{\sph^1} r_n^4   \right) + o(\mu_n^2).  \label{expansion proof intermed}
\end{align}
It remains to find a lower bound which shows that the right side is strictly positive for $n$ large enough. 

We expand $\eta_n$ in spherical harmonics
\[ \eta_n = \sum_{k = 3}^\infty a_{k,n} \sin k\theta + b_{k,n} \cos k \theta. \]
Up to applying an additional rotation, we may assume for simplicity that $\theta_n = 0$ in the decomposition of  Lemma \ref{lemma r + eta}, i.e., $r_n(\theta) = \sin 2 \theta$. Since $\sin^2 2 \theta = \frac{1}{2} - \frac{1}{2} \cos 4 \theta$, we have
\[ \int_{\sph^1} r_n^2 \eta_n = -\frac{\pi}{2} b_{4,n} \]
because all other integrals of $\sin^2 2\theta$ against the $\sin k \theta$ and $\cos k \theta$ with $k \geq 3$ vanish. 
 Following this observation, we may further decompose 
\[ \eta_n = b_{4,n} \cos 4 \theta + \tilde{\eta}_n. \]
Then, recalling $c_{BE}^{\text{loc}}(s) = 1 - \frac{\alpha(1)}{\alpha(2)}$ and $\tilde{\eta}_n \in E_{\geq 3}$, 

\begin{align*}
& \qquad (1 - c_{BE}^{\text{loc}}(s)  + o(1)) (\eta_n, P_s \eta_n) - \alpha(1) \left(\int_{\sph^1} \eta_n^2 - (p-2) \mu_n \int_{\sph^1} r_n^2 \eta_n \right) \\
&= \left( \frac{\alpha(1)}{\alpha(2)} + o(1) \right) (\tilde{\eta}_n, P_s \tilde{\eta}_n) - \alpha(1) \int_{\sph^1} \tilde{\eta}_n^2  +  \alpha(1) \left(\frac{\alpha(4)}{\alpha(2)} - 1 + o(1)\right) b_{4,n}^2 \pi - \alpha(1) (p-2) \mu_n \frac{\pi}{2} b_{4,n} \\
& \geq \alpha(1) \left( \frac{\alpha(3)}{\alpha(2)} - 1 + o(1) \right) \int_{\sph^1} \tilde{\eta}_n^2 \\
& \qquad  +  \alpha(1) \left(\frac{\alpha(4)}{\alpha(2)} - 1 + o(1)\right) \pi \left( b_{4,n}^2 - 2 \frac{p-2}{4} \left(\frac{\alpha(4)}{\alpha(2)} - 1 + o(1)\right)^{-1} \mu_n b_{4,n} \right).   \\
\end{align*}
Now we drop the first summand, which is nonnegative. In the second summand we complete the square in $b_{4,n}$ to obtain
\begin{align*}
& \quad (1 - c_{BE}^{\text{loc}}(s)  + o(1)) (\eta_n, P_s \eta_n) - \alpha(1) \left(\int_{\sph^1} \eta_n^2 - (p-2) \mu_n \int_{\sph^1} r_n^2 \eta_n \right)  \\
&\geq  \alpha(1) \left(\frac{\alpha(4)}{\alpha(2)} - 1 + o(1)\right) \pi \, \times \\
& \qquad \times  \left( \left( b_{4,n} - \frac{p-2}{4} \left(\frac{\alpha(4)}{\alpha(2)} - 1 + o(1)\right)^{-1} \mu_n \right)^2 - \frac{(p-2)^2}{16} \left(\frac{\alpha(4)}{\alpha(2)} - 1 + o(1)\right)^{-2} \mu_n^2   \right)  \\
& \geq - \alpha(1) \left(\frac{\alpha(4)}{\alpha(2)} - 1 + o(1)\right)^{-1} \pi \frac{(p-2)^2}{16} \mu_n^2. 
\end{align*}

Moreover, we can further simplify the term of \eqref{expansion proof intermed} that is quartic in $r_n$ by observing 
\[ \int_{\sph^1} r_n^2  = \pi  \quad \text{ and } \quad  \int_{\sph^1} r_n^4 = \int_0^{2\pi} \sin^4 \theta = \frac{3\pi}{4}. \]
Inserting all of this into \eqref{expansion proof intermed}, we obtain 
\begin{align*}
& \qquad (r_n, P_s r_n) \left( \mathcal E(u_n) - c_{BE}^{\text{loc}}(s) \right) \\
 &\geq \left( - \alpha(1) \left(\frac{\alpha(4)}{\alpha(2)} - 1 + o(1)\right)^{-1} \pi \frac{(p-2)^2}{16} + \alpha(1) \frac{p-2}{12} \left( \frac{3(p-1)}{2\pi} \pi^2 - (p-3) \frac{3\pi}{4} \right) \right) \mu_n^2 \\
&= \alpha(1) \pi \frac{p-2}{16} \left( p + 1 -  \frac{p-2}{\frac{\alpha(4)}{\alpha(2)} - 1 + o(1)} \right) \mu_n^2 .
\end{align*}
Recalling that $p = \frac{2}{1-2s}$ and $\alpha(\ell) = \frac{\Gamma(\ell + \frac{1}{2} + s)}{\Gamma(\ell + \frac{1}{2} - s)}$, an explicit computation gives 
\begin{align*}
 p + 1 -  \frac{p-2}{\frac{\alpha(4)}{\alpha(2)} - 1 + o(1)} &= \frac{3 - 2s}{1 - 2s} - \frac{(5 - 2s)(7 - 2s)}{12(1 - 2s)} + o(1) = \frac{1 - 4s^2}{12(1 - 2s)} + o(1) \\
 &= \frac{1 + 2s}{12}+ o(1)  . 
\end{align*} 
As a consequence, 
\[ \mathcal E(u_n) - c_{BE}^{\text{loc}}(s) \geq  (r_n, P_s r_n)^{-1}  \alpha(1)\pi \frac{p-2}{16} \frac{1 + 2s}{12} \mu_n^2+ o(\mu_n^2) > 0 \]
for every $n$ large enough. This finishes the proof (in the easy case $p \geq 4$). 
\end{proof}

Let us now explain how to drop the assumption \eqref{p geq 4} which states that $p \geq 4$. If $ p < 4$, the very first step in the preceding proof is not justified, namely expanding $|1 + \rho_n|^p$ to fourth order, because $t \mapsto |t|^p$ is not four times continuously differentiable in $0$. 

That means, the fourth order expansion of $|1 + \rho_n(\theta)|^p$ is only justified at points $\theta$ where $\rho_n(\theta) > -1$. For this condition to be fulfilled for all $\theta \in (0,2\pi)$, we would for example need $\rho_n$ (or equivalently $\eta_n$) to converge to zero \emph{uniformly} on $\sph^1$. However, since $H^s(\sph^1)$ does not embed into $L^\infty(\sph^1)$, this is not necessarily the case. To overcome this problem we adapt and simplify a strategy carried out in \cite{Frank2023} in a similar situation for $s = 1$ and $d \geq 1$. 

\begin{proof}
[Proof of Theorem \ref{theorem d=1 Sd}, hard case $p < 4$]
Let again a sequence $u_n = 1 + \rho_n$ be fixed which satisfies \eqref{E(u_n) to cloc} and \eqref{u = 1 + rho}. Notice that Lemma \ref{lemma r + eta} holds for $u_n$ also in this case. 

 In view of the discussion above, we denote  
\begin{equation}
\label{Cn definition}
\mathcal C_n := \left\{ \theta \in S^1 \, : \, |\rho_n(\theta)| > \frac 12 \right\}.
\end{equation}
On $\sph^1 \setminus \mathcal C_n$, we have $\rho_n \in [-\frac{1}{2}, \frac{1}{2}]$ and therefore we can expand, similarly to the above, 
\begin{align}
\int_{\sph^1 \setminus \mathcal C_n} |1 + \rho_n|^p &= \int_{\sph^1 \setminus \mathcal C_n} 1 + p \int_{\sph^1 \setminus \mathcal C_n} \rho_n + \frac{p(p-1)}{2} \int_{\sph^1 \setminus \mathcal C_n} \rho_n^2 + \frac{p(p-1)(p-2)}{6} \int_{\sph^1 \setminus \mathcal C_n} \rho_n^3 \nonumber \\
& \qquad +  \frac{p(p-1)(p-2)(p-3)}{24} \int_{\sph^1 \setminus \mathcal C_n} \rho_n^4 + \mathcal O\left( \int_{\sph^1 \setminus \mathcal C_n} |\rho_n|^5 \right).   \label{int s1 setminus Cn}
\end{align}
On the other hand, since $t \mapsto |t|^p$ is  twice differentiable on $\R$, on $\mathcal C_n$ we can still expand to second order, 
\begin{equation}
\label{int Cn}
\int_{\mathcal C_n} |1 + \rho_n|^p = \int_{\mathcal C_n} 1 + p \int_{\mathcal C_n} \rho_n + \frac{p(p-1)}{2} \int_{\mathcal C_n} \rho_n^2 + \mathcal O \left(\int_{\mathcal C_n} |\rho_n|^{\min\{3, p\}} + |\rho_n|^p \right). 
\end{equation} 
Let us now observe that we can bound the measure of $\mathcal C_n$, uniformly in $n \in \N$, by
\begin{equation}
\label{Cn bound}
|\mathcal C_n| \lesssim \mu_n^p. 
\end{equation} 
Indeed, this follows from 
\[ |\mathcal C_n| \left( \frac{1}{2} \right)^p \leq \int_{\sph^1} |\rho_n|^p \lesssim \mu_n^p \int_{\sph^1} |r_n|^p + \mu_n^p \int_{\sph^1} |\eta_n|^p \lesssim |\mu_n|^p \]
by Sobolev's inequality and Lemma \ref{lemma r + eta}. 

Using \eqref{Cn bound} and the fact that $r_n$ is uniformly bounded, we can bound the error terms in \eqref{int Cn} by 
\[ \int_{\mathcal C_n} |\rho_n|^p \lesssim \mu_n^p \int_{\mathcal C_n} |r_n|^p  + \mu_n^p \int_{\mathcal C_n} |\eta_n|^p \lesssim \mu_n^p |\mathcal C_n| + \mu_n^p \|\eta_n\|^p = o(\mu_n^4 + \mu_n^2 \|\eta_n\|^2). \]
Moreover, in the same way we can estimate 
\[  \int_{\mathcal C_n} |\rho_n|^3 \lesssim  \mu_n^{3 + p} + \mu_n^3 \|\eta_n\|^3 = o(\mu_n^4 + \mu_n^2 \|\eta_n\|^2) \]
and 
\[  \int_{\mathcal C_n} |\rho_n|^4 \lesssim \mu_n^{4 + p} + \mu_n^4 \|\eta_n\|^4 = o(\mu_n^4 + \mu_n^2 \|\eta_n\|^2). \]
Finally, the error term in \eqref{int s1 setminus Cn}, using that $\mu_n |\eta_n| \leq |\rho_n| + \mu_n |r_n| \leq 1$, can be bounded by 
\[    \int_{\sph^1 \setminus \mathcal C_n} |\rho_n|^5 \lesssim \mu_n^5 \int_{\sph^1 \setminus \mathcal C_n} |r_n|^5 + \mu_n^5 \int_{\sph^1 \setminus \mathcal C_n} |\eta_n|^5 \lesssim \mu_n^5 + \mu_n^p \int_{\sph^1 \setminus \mathcal C_n} |\eta_n|^p = o(\mu_n^4 + \mu_n^2 \|\eta_n\|^2). \]
With all these error estimates, we have proved that by adding up \eqref{int s1 setminus Cn} and \eqref{int Cn} we obtain
\begin{align*}
\int_{\sph^1} |1 + \rho_n|^p &=  \int_{\sph^1} 1 + p \int_{\sph^1} \rho_n + \frac{p(p-1)}{2} \int_{\sph^1} \rho_n^2 + \frac{p(p-1)(p-2)}{6} \int_{\sph^1} \rho_n^3  \\
& \qquad +  \frac{p(p-1)(p-2)(p-3)}{24} \int_{\sph^1} \rho_n^4 +  o(\mu_n^4 + \mu_n^2 \|\eta_n\|^2).
\end{align*}
Now we can decompose $\rho_n = \mu_n (r_n + \eta_n)$ and use some estimates we have already explained above. In this way we obtain the analogue of \eqref{proof basic expansion}. From there we can proceed with the proof as in the 'easy case' $p \geq 4$. 
\end{proof}

\section{A family of test functions for $\mathcal E(u)$}
\label{section conjecture}

In this section, we study in some detail a family $(u_\beta)$ of natural test functions which interpolates between one bubble and two bubbles, and which we define in Subsection \ref{subsection bubbles properties} below. Here, actually most of our analysis will be carried out for general dimension $d \geq 1$. 

Our analysis of $(u_\beta)$ leads to several interesting consequences. Firstly, for $d \geq 2$ the $(u_\beta)$ yield a different and very natural choice of test function that gives the strict inequality $c_{BE}(s) < c_{BE}^{\text{loc}}(s)$ originally proved in \cite{Koenig}. See Subsection \ref{subsection strict ineq} below.

Secondly, for $d = 1$, through some computations and basic numerics carried out in Subsections \ref{subsection p=3} and \ref{subsection p=4} below we show that $\mathcal E(u_\beta) > c_{BE}^\text{loc}(s)$ for all values of $\beta$. Since the $u_\beta$ can be expected to be very good competitors for the global infimum $c_{BE}(s)$ (see the discussion in Subsection \ref{subsection bubbles properties}), this provides some more evidence which supports our Conjecture \ref{conjecture BE}.

\subsection{Sums of two bubbles and their properties}
\label{subsection bubbles properties}

For any $d \geq 1$, $s \in (0, d/2)$ and $p = \frac{2d}{d-2s} \in (2, \infty)$, let
\begin{equation}
\label{v beta definition}
v_\beta (\omega) := (1 - \beta^2)^\frac{d}{2p} (1 - \beta \omega_{d+1})^{-\frac{d}{p}}, \qquad \beta \in (-1, 1). 
\end{equation}
The function $v_\beta$ is an optimizer of the Sobolev inequality \eqref{sobolev}, i.e., $v_\beta \in \mathcal M$. Its normalization is chosen such that 
\begin{equation}
\label{v beta normalization}
\int_{\mathbb S^d} v_\beta^p \, d \omega = \int_{\mathbb S^d} 1 = |\mathbb S^d| \quad \text{ and } \quad (v_\beta, P_s v_\beta) = (1, P_s 1) = \alpha(0) |\mathbb S^d|. 
\end{equation}
Under the stereographic projection defined in \eqref{ster proj definition} and \eqref{conf trafo}, the family $(v_\beta)_{\beta \in (-1, 1)}$ corresponds precisely to all dilations $B_\lambda(x) = \lambda^\frac{d}{p} B(\lambda x)$, $\lambda > 0$, of the standard bubble $B(x) = \left(\frac{2}{1+ |x|^2}\right)^\frac{d}{p}$ centered in $0 \in \R^d$. 

Let us now consider the family made of sums of two bubbles $v_\beta$ given by
\begin{equation}
\label{u beta definition}
u_\beta := v_\beta + v_{-\beta} 
\end{equation} 
and note that by symmetry it is sufficient to consider the range $\beta \in (0, 1)$. In the equivalent setting of $\R^d$, the family $u_\beta$ interpolates between one bubble centered at the origin having twice the $L^p$-norm of $B$ (for $\beta = 0$) and the superposition $B_\lambda + B_{\lambda^{-1}}$ of two weakly interacting bubbles centered in $0$, with $\lambda = \lambda(\beta) = \sqrt{\frac{1+\beta}{1 - \beta}} \to \infty$ as $\beta \to 1$. 

Let us discuss in some more detail the reasons why, heuristically, we expect the $(u_\beta)$ to be good competitors for the global infimum $c_{BE}(s)$, i.e.  $\inf_{\beta \in (0,1)} \mathcal E(u_\beta)$ to give a value close to $c_{BE}(s)$. 

Firstly, it is shown in \cite{Koenig-min} that the asymptotic values of $\mathcal E(u_\beta)$ as $\beta \to 1$ is
\begin{equation}
\label{c 2 bubble}
\lim_{\beta \to 1} \mathcal E(u_\beta) = 2 - 2^\frac{d-2s}{d} .
\end{equation}
Moreover, the analysis of minimizing sequences from \cite{Koenig-min} shows that this value is best possible for sequences consisting of at least two non-trivial asymptotically non-interacting parts. 

Secondly, as $\beta \to 0$, the quotient $\mathcal E(u_\beta)$ converges to the best local constant $c_{BE}^{\text{loc}}(s)$. What is more, for every $d \geq 2$, it even does so \emph{from below}. This is the content of Proposition \ref{proposition local two-bubbles} below.  

The third reason why we expect the $u_\beta$ to be good competitors for $c_{BE}(s)$ concerns the whole range $\beta \in (0,1)$, not just the asymptotic regime. Indeed, the shape of $u_\beta$ as $\beta$ varies reflects the two competing required properties of a minimal function $u$ for $\mathcal E$. On the one hand, the numerator of $\mathcal E$, i.e., the Sobolev deficit $(u, P_s u) - \mathcal S_d \|u\|_p^2$ should be small, hence $u$ should look like a Talenti bubble (which is the case for $\beta$ small). On the other hand, the denominator of $\mathcal E$, i.e., the distance $\dist(u, \mathcal M)^2$, should be large, which forces $u$ to be different from a Talenti bubble. The family $(u_\beta)$ represents a natural attempt to reconcile these two competing necessities. 

In connection with this, we can mention an interesting analogy with the stability inequality associated to the isoperimetric inequality, whose features are similar to inequality \eqref{bianchi egnell}. For $d=2$, a conjecture which is strongly supported by both numerics and partial rigorous arguments states that the optimal set for the isoperimetric stability inequality is a certain explicit non-convex mask-shaped set, see \cite{BiCrHe}. This set is precisely a compromise between one and two disjoint balls, in much the same way as $u_\beta$, for intermediate values of $\beta \in (0,1)$ is a (probably not fully optimal) compromise between one and two non-interacting bubbles. 

\subsection{An alternative proof for $c_{BE}(s) < c_{BE}^{\text{loc}}(s)$ when $d \geq 2$}
\label{subsection strict ineq}

The following proposition yields an alternative proof of the main result of the recent paper \cite{Koenig}, namely the fact that $c_{BE}(s) < c_{BE}^\text{loc}(s)$ for every $d \geq 2$. 

\begin{proposition}
\label{proposition local two-bubbles}
Let $d \geq 1$ and let $(u_\beta)_{\beta \in (-1,1)}$ be the family of functions defined in \eqref{u beta definition}. Then there are $c_1(\beta), c_2 > 0$ such that $c_1(\beta) \to 2$ and 
\begin{equation}
\label{u beta expansion prop} 
u_\beta = c_1(\beta) + c_2 \beta^2 \rho + o(\beta^2)
\end{equation}
uniformly on $\mathbb S^d$ as $\beta \to 0$, where 
\begin{equation}
\label{rho definition}
\rho(\omega) = \omega_{d+1}^2 - \frac{1}{d} \sum_{i=1}^d \omega_i^2  \quad \in E_2. 
\end{equation}
As a consequence, 
 $\lim_{\beta \to 0} \mathcal E(u_\beta) = \frac{4s}{d + 2s +2} = c_{BE}^\text{loc}(s)$. 
Moreover, if $d \geq 2$,
\begin{equation}
\label{strict ineq proposition}
\mathcal E(u_\beta) < c_{BE}^\text{loc} (s)
\end{equation} 
for all $\beta$ small enough. 
\end{proposition}

As we will see in the proof of Proposition \ref{proposition local two-bubbles}, the proof of the strict inequality \eqref{strict ineq proposition} boils down, via a Taylor expansion, to proving that $\int_{\mathbb S^d} \rho^3 > 0$ if $d \geq 2$ (while $\int_{\mathbb S^1} \rho^3  =0$ if $d = 1$). Exhibiting a second spherical harmonic $\rho \in E_2$ with this property has been the key observation of \cite{Koenig}, however the spherical harmonic chosen there is different from $\rho$ in \eqref{rho definition}.  Arguably, the choice of $\rho$ in \eqref{rho definition} is more natural, because it comes from the family $(u_\beta)$ via \eqref{u beta expansion prop}, while the choice in \cite{Koenig} is made on abstract and purely algebraic grounds. It must be noted that both choices require $d \geq 2$ for \eqref{strict ineq proposition} to hold, which is consistent with Theorem \ref{theorem d=1}. 

\begin{proof}
[Proof of Proposition \ref{proposition local two-bubbles}]
The proof of \eqref{u beta expansion prop} comes from a straightforward Taylor expansion. Indeed, 
\begin{align*}
v_\beta(\omega) &= (1- \beta^2)^\frac{d}{2p} (1 - \beta \omega_{d+1})^{-\frac{d}{p}} \\
& = \left(1 - \frac{d}{2p} \beta^2 + o(\beta^2)\right) \left( 1 + \frac{d}{p}\beta \omega_{d+1} + \frac{1}{2} \frac{d}{p}\left(\frac{d}{p} +1 \right) \beta^2 \omega_{d+1}^2 + o(\beta^2) \right) \\
&=\left(1 - \frac{d}{2p} \beta^2 \right) + \frac{d}{p} \beta \omega_{d+1}  + \frac{1}{2} \frac{d}{p} \left(\frac{d}{p} + 1 \right) \beta^2 \omega_{d+1}^2 + o(\beta^2). 
\end{align*}
Hence 
\[ u_\beta = v_\beta + v_{-\beta} = \left(2 - \frac{d}{p} \beta^2 \right) + \frac{d}{p} \left(\frac{d}{p} + 1 \right) \beta^2 \omega_{d+1}^2 + o(\beta^2). \]
From this we conclude \eqref{u beta expansion prop} by observing that 
\[ \omega_{d+1}^2 = \frac{d}{d+1} \rho + \frac{1}{d+1}, \]
with $\rho$ defined by \eqref{rho definition}. 

Now we turn to proving \eqref{strict ineq proposition}. Given \eqref{u beta expansion prop} and the fact that $\rho \in E_2$, it now follows from a Taylor expansion of the quotient $\mathcal E(u_\beta)$ that 
\[ \mathcal E(u_\beta) = \mathcal E\left(1 + \frac{c_2}{2} \beta^2 \rho \right) + o(\beta^2) = c_{BE}^\text{loc}(s) - c_3 \beta^2 \int_{\mathbb S^d} \rho^3 + o(\beta^2) \]
for some constant $c_3 > 0$, whose explicit value is of no interest to us. This follows from the computations made in \cite{Koenig}, respectively from their equivalent on $\mathbb S^d$ via stereographic projection. 

To prove \eqref{strict ineq proposition}, it therefore only remains to show that $\int_{\mathbb S^d}\rho^3 > 0$, where $\rho$ is given by \eqref{rho definition}. We compute 
\begin{align*}
\int_{\mathbb S^d}\rho^3 &= |\mathbb S^{d-1}| \int_0^{\pi} \left( \cos^2 \theta - \frac{1}{d} \sin^2 \theta \right)^3 \sin^{d-1} \theta \, d \theta \\
&= |\mathbb S^{d-1}| \int_0^{\pi} \left( \cos^6 \theta - \frac{3}{d} \cos^4 \theta \sin^2 \theta + \frac{3}{d^2} \cos^2 \theta \sin^4 \theta - \frac{1}{d^3} \sin^6 \theta \right) \sin^{d-1} \theta \, d \theta. 
\end{align*}
Integration by parts yields the relation
\[ \int_0^\pi \cos^k \theta \sin^l \theta \, d \theta = \frac{k-1}{l+1} \int_0^\pi \cos^{k-2} \theta \sin^{l+2} \theta \, d \theta \]
for all $k \geq 2$, $l \geq 0$. Applying this repeatedly, a straightforward calculation leads to 
\[ \int_{\mathbb S^d}\rho^3 = |\mathbb S^{d-1}| \frac{8 (d^2 - 1)}{d^3 (d +2)(d+4)} \int_0^\pi \sin^{d+5} \theta  \, d \theta. \]
Thus for $\rho$ given by \eqref{rho definition}, $\int_{\mathbb S^d} \rho^3$ is strictly positive whenever $d \geq 2$ (and zero when $d = 1$). Hence the proof is complete. 
\end{proof}

\subsection{The case $p = 3$}
\label{subsection p=3}

We now turn to evaluating the family $(u_\beta)$ for intermediate values of $\beta$. This proves to be much harder than obtaining asymptotic values.  One of the reasons for this is the fact that there is some $\beta_0 \in (0,1)$ such that for $\beta > \beta_0$ the distance $\dist(u_\beta, \mathcal M)$ is no longer achieved by a constant. To simplify this particular issue, we introduce the modified Bianchi--Egnell quotient 
\begin{equation}
\label{BE tilde}
\widetilde{\mathcal E}(u) := \frac{(u, P_s u) - S_{d,s} \|u\|_p^2}{\dist_{H^s(\mathbb S^d)}(u, \mathcal C)^2}.
\end{equation}
The difference of $\widetilde{\mathcal E}(u)$ to $\mathcal E(u)$ is that the denominator in \eqref{BE tilde} contains the ${H}^s$-distance to the set $\mathcal C \subset \mathcal M$ of  constant functions, instead of all optimizers $\mathcal M$. The advantage of $\tilde{\mathcal{E}}(u)$ for computations is that the function $c \in \mathcal C$ realizing $\dist_{H^s(\mathbb S^d)}(u, \mathcal C)$ can be determined very easily for every function $u \in H^s(\mathbb S^d)$, while this is not so for $\dist_{H^s(\mathbb S^d)}(u, \mathcal M)$. 

Since $\mathcal C \subset \mathcal M$, we moreover have 
\begin{equation}
\label{E geq E tilde}
\mathcal E(u) \geq  \widetilde{\mathcal E}(u) \qquad \text{ for all } u \in H^s(\mathbb S^d).
\end{equation}
Finally, for small enough $\beta$ we actually have $\mathcal E(u_\beta) =\widetilde{\mathcal E}(u_\beta)$: this follows from \cite[Lemma 12]{Frank2023b} together with the fact that for the distance minimizer $c_\beta$ of $\dist_{H^s(\mathbb S^d)}(u_\beta, \mathcal C)$, one has $\int_{\mathbb S^d} (u_\beta - c_\beta) = \int_{\mathbb S^d} (u_\beta - c_\beta) \omega = 0$, for all $\beta \in (0,1)$. (However, for $\beta$ close to 1 the minimizer of $\dist_{H^s(\mathbb S^d)}(u_\beta, \mathcal M)$ must be close to $v_\beta$ or $v_{-\beta}$, and hence $\mathcal E(u_\beta) > \widetilde{\mathcal E}(u_\beta)$ for such $\beta$.)

To simplify computations further, we only consider a special choice of $p$ which gives an algebraically simple expression, namely $p = 3$. We will make some largely analogous computations for $p = 4$ in the next subsection.

Our goal in this subsection is thus to confirm that for $d = 1$ and $p = 3$ (i.e., $s = \frac{1}{6}$) we have 
\begin{equation}
\label{ineq p=3}
\widetilde{\mathcal E}(u_\beta) > c_{BE}^{\loc}(\frac{1}{6}) = \frac{1}{5} \qquad \text{ for all } \beta \in (0,1). 
\end{equation}
By \eqref{E geq E tilde}, this implies in particular $\mathcal E(u_\beta) > c_{BE}^{\loc}(\frac{1}{6})$. 

Unfortunately, it turns out that we are only able to prove \eqref{ineq p=3} \emph{numerically}. It would be desirable to obtain a mathematically rigorous proof of \eqref{ineq p=3} and to extend \eqref{ineq p=3} to all values of $s \in (0, \frac{1}{2})$, either numerically or rigorously.

We emphasize once more that for $d \geq 2$ property \eqref{ineq p=3} must fail for small enough $\beta$, as a consequence of Proposition \ref{proposition local two-bubbles}. Nevertheless, we carry out the following computations for general $d \geq 1$, because they present no additional difficulty and because the results may be of some independent use. 

In the following lemma, we express $\widetilde{\mathcal E}(u_\beta)$ conveniently in terms of the quantity
\begin{equation}
\label{I beta definition}
I(\beta):= \frac{1}{|\mathbb S^d|} \int_{\mathbb S^d} v_\beta.
\end{equation}  

\begin{lemma}
\label{lemma tilde E}
Let $p = 3$ and $d \geq 1$, so that $s = \frac{d}{6}$. Then for every $\beta \in (0,1)$, we have 
\[ \widetilde{\mathcal E}(u_\beta)= \frac{1 + I(\gamma(\beta)) - 2^{-1/3} (1 + 3 I(\gamma(\beta)))^{2/3}}{1 + I(\gamma(\beta)) - 2 I(\beta)^2}, \]
where $\gamma(\beta) = \frac{2 \beta}{1 + \beta^2}$. 
\end{lemma}

Before giving the proof of this lemma, we make a useful observation which explains the origin of the expression for $\gamma(\beta)$. 

\begin{lemma}
\label{lemma gamma}
For $d \geq 1$, let $s \in (0, \frac{d}{2})$ and $p = \frac{2d}{d-2s} > 2$. For every $\beta, \beta' \in (-1, 1)$, we have 
\begin{equation}
\label{beta gamma identity lemma}
(v_\beta, P_s v_{\beta'}) = (v_\gamma, P_s 1) \qquad \text{ and } \qquad \int_{\mathbb S^d} v_\beta^{p-1} v_{\beta'} = \int_{\mathbb S^d} v_\beta v_{\beta'}^{p-1} = \int_{\mathbb S^d}  v_{\gamma} , 
\end{equation} 
where 
\[ \gamma = \gamma(\beta, \beta') = \frac{\beta - \beta'}{1 - \beta \beta'}. \]
In particular, if $\beta' = - \beta$, then $\gamma$ and $\beta$ are related by
\[ \gamma(\beta) := \gamma(\beta, -\beta) = \frac{2 \beta}{1 + \beta^2} \qquad \Leftrightarrow \qquad \beta = \frac{1 - \sqrt{1 - \gamma^2}}{\gamma} \]
(Here, if $\gamma = 0$, we interpret  $ \frac{1 - \sqrt{1 - \gamma^2}}{\gamma} = 0$.)
\end{lemma}

\begin{proof}
Let $B(x) = \left( \frac{2}{1 + |x|^2} \right)^d$ and denote $B_\lambda(x) = \lambda^\frac{d}{p} B(\lambda x)$. 

A direct computation then gives that 
\[ (v_\beta)_{\mathcal S} = B_{\lambda}, \]
where the transformation $u \mapsto u_{\mathcal S}$ is defined by \eqref{conf trafo}, and $\lambda$ and $\beta$ are related by
\begin{equation}
\label{beta lambda relation}
\lambda = \lambda(\beta) = \sqrt{\frac{1+\beta}{1 - \beta}} \qquad \Leftrightarrow \qquad  \beta = \beta(\lambda) = \frac{\lambda^2-1}{\lambda^2+1}. 
\end{equation} 
By conformal invariance, we have
\[ (v_\beta, P_s v_{\beta'}) = (B_{\lambda(\beta)}, (-\Delta)^s B_{\lambda(\beta')})_{L^2(\R^d)} = (B_{\lambda(\beta) \lambda(\beta')^{-1}}, (-\Delta)^s B)_{L^2(\R^d)}  = (v_{{\gamma}}, P_s 1), \]
where, by \eqref{beta lambda relation}, $\gamma$ is given by 
\[ {{\gamma}} =  \frac{\lambda(\beta)^2 \lambda(\beta')^{-2} - 1}{\lambda(\beta)^2 \lambda(\beta')^{-2} +1} =  \frac{\frac{1+\beta}{1 - \beta} \frac{1-\beta'}{1 + \beta'} - 1}{\frac{1+\beta}{1 - \beta} \frac{1-\beta'}{1 + \beta'} + 1} = \frac{(1 + \beta) (1 - \beta') - (1 - \beta) (1 + \beta')}{(1 + \beta) (1 - \beta') + (1 - \beta) (1 + \beta')} = \frac{\beta - \beta'}{1 - \beta \beta' }  \]
as claimed. The second identity in \eqref{beta gamma identity lemma} follows from this by using that $P_s v_\beta = \alpha(0) v_\beta^{p-1}$ for every $\beta \in (-1, 1)$. 

The expression $\beta = \frac{1 - \sqrt{1 - \gamma^2}}{\gamma}$ comes from solving the quadratic equation $(1 + \beta^2) \gamma = 2 \beta$ and taking into account that both $\gamma$ and $\beta$ are in $(-1,1)$. 
\end{proof}

\begin{proof}
[Proof of Lemma \ref{lemma tilde E}]
We compute the terms in $\widetilde{\mathcal E}(u_\beta)$ separately, making repeated use of the normalization \eqref{v beta normalization} and Lemma \ref{lemma gamma}. We have
\begin{align*}
(u_\beta, P_s u_\beta) &= 2 (1, P_s 1) + 2 (v_{-\beta}, P_s v_\beta) =  2 |\mathbb S^d| \alpha(0) + 2 (v_{\gamma(\beta)}, P_s 1) \\
& = 2 |\mathbb S^d| \alpha(0) + 2 \alpha(0) \int_{\mathbb S^d} v_{\gamma(\beta)} = 2 |\mathbb S^d| \alpha(0) \left(1 + I(\gamma(\beta))\right), 
\end{align*} 

Similarly, we obtain (here is where we use $p=3$ to simply multiply out the terms!)
\[ \int_{\mathbb S^d} u_\beta^3 = 2 \int_{\mathbb S^d} 1 + 6 \int_{\mathbb S^d} v_\beta^2 v_{-\beta} = 2 |\mathbb S^d| + 6 |\mathbb S^d|  I(\gamma(\beta)). \]
Finally, it is easy to see that $\dist (u_\beta, \mathcal C)$ is uniquely achieved by the constant function $c_\beta$ which satisfies $\int_{\mathbb S^d} c_\beta = \int_{\mathbb S^d} u_\beta = 2 \int_{\mathbb S^d} v_\beta$, i.e.,
\begin{equation}
\label{c beta}
c_\beta = \frac{2}{|\mathbb S^d|} \int_{\mathbb S^d} v_\beta = 2 I(\beta). 
\end{equation} 
Therefore 
\begin{align*}
\dist(u_\beta, \mathcal C)^2 &= (u_\beta - c_\beta, P_s (u_\beta - c_\beta)) = (u_\beta, P_s u_\beta) - 2 (u_\beta, P_s c_\beta) + (c_\beta, P_s c_\beta) \\
&= (u_\beta, P_s u_\beta) - 2 \alpha(0) c_\beta \int_{\mathbb S^d} u_\beta + |\mathbb S^d| \alpha(0) c_\beta^2  \\
&= (u_\beta, P_s u_\beta) - |\mathbb S^d| \alpha(0) c_\beta^2 \\
&=2 |\mathbb S^d| \alpha(0) \left(1 + I(\gamma(\beta))\right) - 4  |\mathbb S^d|  \alpha(0) I(\beta)^2.
\end{align*}

Combining everything, and observing $S_{d,s} = \alpha(0) |\mathbb S^d|^{1 - \frac{2}{p}} = \alpha(0) |\mathbb S^d|^{\frac{1}{3}}$ since  $p=3$, we obtain 
\begin{align*}
\widetilde{\mathcal E}(u_\beta) &= \frac{(u_\beta, P_s u_\beta) - S_{d,s} \|u_\beta\|_3^2}{\dist(u_\beta, \mathcal C)^2} \\
&= \frac{2 |\mathbb S^d|  \alpha(0) \left(1 + I(\gamma(\beta))\right) - \alpha(0)|\mathbb S^d| ^{1/3} \left(2 |\mathbb S^d|  + 6 |\mathbb S^d|  I(\gamma(\beta))\right)^{2/3} }{2 |\mathbb S^d|  \alpha(0) (1 + I(\gamma(\beta))) - 4 |\mathbb S^d|  \alpha(0) I(\beta)^2} \\
&= \frac{1 + I(\gamma(\beta)) - 2^{-1/3} \left(1 + 3 I(\gamma(\beta))\right)^{2/3}}{1 + I(\gamma(\beta)) - 2 I(\beta)^2},
\end{align*}
which is the claimed expression. 
\end{proof}

 To obtain a plot of the values of $\widetilde{\mathcal E}(u_\beta)$, we now express $I(\beta)$ in terms of a suitable hypergeometric function $_2F_1(a,b; c; z)$. A useful reference for the definition and properties of  $_2F_1(a,b; c; z)$ is \cite[Chapter 15]{AbSt}. In any case, all we need to know about the function $_2F_1(a,b; c; z)$ is the identity stated in \eqref{hypergeo int rep} below. 
 
In fact, with no extra work we can obtain an expression for 
\begin{equation}
\label{I q beta definition}
I_q(\beta) := \frac{1}{|\mathbb S^d|} \int_{\mathbb S^d} v_\beta^q, 
\end{equation} 
for any $q > 0$, which we will use in the next subsection for $q = 2$ as well.  

\begin{lemma}
\label{lemma hypergeo}
Let $d \geq 1$ and $p > 2$. For $\beta \in (0,1)$ and $q > 0$, let $I_q(\beta)$ be given by \eqref{I q beta definition}. Then 
\[ I_q(\beta) = c_d \left( \frac{1 - \beta}{1 + \beta} \right)^\frac{dq}{2p} {_2F_1} \left( \frac{dq}{p}, \frac{d}{2} ; d ; \frac{2 \beta}{1+ \beta} \right), \]
where $c_d = 2^{d-1} \frac{|\mathbb S^{d-1}|}{|\mathbb S^{d}|} \frac{\Gamma(d/2)^2}{\Gamma(d)} =    2^{d-1} \frac{\Gamma(\frac{d}{2}) \Gamma(\frac{d+1}{2})}{\Gamma(\frac{1}{2}) \Gamma(d)}$. 

With $z = \frac{2 \beta}{1 + \beta}$, we can rewrite this as
\[ I_q(\beta) = c_d (1 - z)^\frac{dq}{2p} {_2F_1} \left( \frac{dq}{p}, \frac{d}{2} ; d ; z \right) =: J(z). \]
\end{lemma}

\begin{proof}
We recall the integral representation of the hypergeometric function, which reads \cite[eq. 15.3.1]{AbSt}
\begin{equation}
\label{hypergeo int rep}
_2F_1(a,b;c;z) = \frac{\Gamma(c)}{\Gamma(b)\Gamma(c-b)} \int_0^1 t^{b-1} (1 -t)^{c - b - 1} (1 - zt)^{-a} \, dt. 
\end{equation}
Thus, for every $\beta \in (0,1)$, we have 
\begin{align*}
I_q(\beta) &= \frac{|\mathbb S^{d-1}|}{|\mathbb S^{d}|} (1 - \beta^2)^\frac{dq}{2p} \int_0^{ \pi} (1 - \beta \cos \theta)^{-\frac{dq}{p}} \sin^{d-1} \theta \, d \theta \\
& =  \frac{|\mathbb S^{d-1}|}{|\mathbb S^{d}|} (1 - \beta^2)^\frac{dq}{2p} \int_{-1}^{1} (1 - \beta t)^{-\frac{dq}{p}} (1 - t^2)^\frac{d-2}{2} \, d t \\
&= \frac{|\mathbb S^{d-1}|}{|\mathbb S^{d}|} 2^{d-1} \left( \frac{1 - \beta}{1 + \beta} \right)^\frac{dq}{2p} \int_0^1  (1 - \frac{2 \beta}{ 1 + \beta} t)^{-\frac{dq}{p}} t^\frac{d-2}{2} (1- t)^\frac{d-2}{2} \, dt. 
\end{align*}
Using \eqref{hypergeo int rep} with $a = dq/p$, $b = d/2$ and $c = d$, and inserting $|\mathbb S^{d-1}| = \frac{2 \pi^{d/2}}{\Gamma(\frac{d}{2})}$ we obtain the claimed expression. 

Finally, for $z = \frac{2 \beta}{1 + \beta}$ we have $\beta = \frac{z}{2-z}$ and hence $\frac{1- \beta}{1 + \beta} = 1 - z$ by direct computation. This yields the claimed expression for $J(z)$. 
\end{proof}

We can now combine Lemmas \ref{lemma tilde E} and \ref{lemma hypergeo} to express $\widetilde{\mathcal E}(u_\beta)$ as an explicit function of one variable, which we can evaluate numerically. 

Indeed, for $\gamma(\beta)= \frac{2\beta}{1 + \beta^2}$, direct computations give $\frac{1 - \gamma}{1+ \gamma} = \left( \frac{1 -\beta}{1 + \beta} \right)^2$ and $\frac{2 \gamma}{1 + \gamma} = \frac{4 \beta}{(1 + \beta)^2}$, so that 
\[ I(\gamma(\beta)) = c_d \left( \frac{1 - \beta}{1 + \beta} \right)^\frac{d}{p} {_2F_1} \left( \frac{d}{p}, \frac{d}{2} ; d ;\frac{4 \beta}{(1+ \beta)^2} \right). \]
In terms of $z = \frac{2 \beta}{1 + \beta}$, we can write this as
\[ I(\gamma(\beta)) = c_d (1 - z)^\frac{d}{p} {_2F_1} \left( \frac{d}{p}, \frac{d}{2} ; d ; z (2 - z) \right). \]
To simplify even further the resulting expression for $\widetilde{\mathcal E}(u_\beta)$ (\eqref{E u beta final} below), it is convenient to rewrite this once again, in terms of the variable $y = z(2-z) \, \Leftrightarrow \, z = 1 - \sqrt{1-y}$, as 
\begin{equation}
\label{m definition}
 I(\gamma(\beta)) = c_d (1 - y)^\frac{d}{2p} {_2F_1} \left( \frac{d}{p}, \frac{d}{2} ; d ; y \right) =: m(y). 
\end{equation}
The expression for $I(\beta)$ becomes accordingly 
\[
I(\beta) = c_d (1 - y)^\frac{d}{4p} {_2F_1} \left( \frac{d}{p}, \frac{d}{2} ; d ; 1 - \sqrt{1 - y} \right) = m(1 - \sqrt{1 - y}). 
\]

Summing up, for $y = z(2-z) = \frac{4\beta}{(1 + \beta)^2}$ (which varies between $0$ and $1$ as $\beta$ does), we have  
\begin{equation}
\label{E u beta final}
\widetilde{\mathcal E}(u_\beta) = \frac{1 + m(y) - 2^{-1/3} (1 + 3 m(y))^{2/3}}{1 + m(y) - 2 m(1 - \sqrt{1-y})^2} =: \frac{e_1(y)}{e_2(y)} =: e(y) .
\end{equation} 

Figure \ref{plot} shows the plot of the function $e(y)$ for dimension $d= 1$.  As expected, it goes to $\frac{1}{5} = c_{BE}^{\text{loc}}(\frac{1}{6})$ as $y \to 0$. Moreover, from the plot, $e(y)$ is manifestly increasing in $y$. Equivalently, $\widetilde{\mathcal E}(u_\beta)$ is increasing in $\beta \in (0,1)$.  This confirms the desired inequality \eqref{ineq p=3}, at least numerically.

\begin{figure}[h]
\centering
\includegraphics[width=0.7\textwidth]{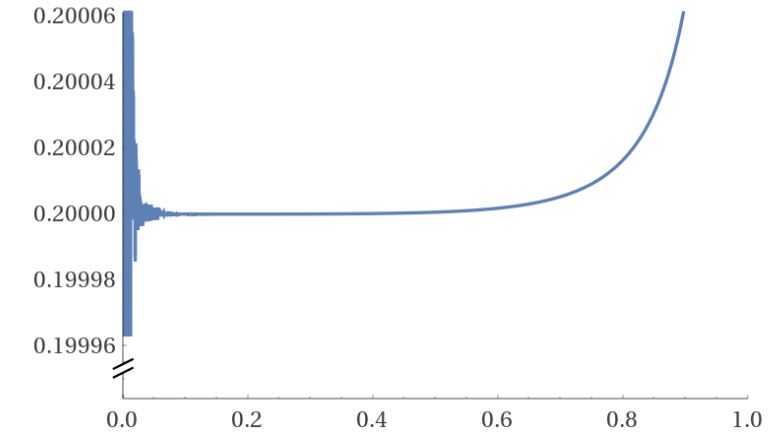}
\caption{Case $p = 3$: Plot of $y \mapsto e(y)$ for $d = 1$. }
\label{plot}
\end{figure}

\begin{figure}[h]
\centering
\includegraphics[width=0.7\textwidth]{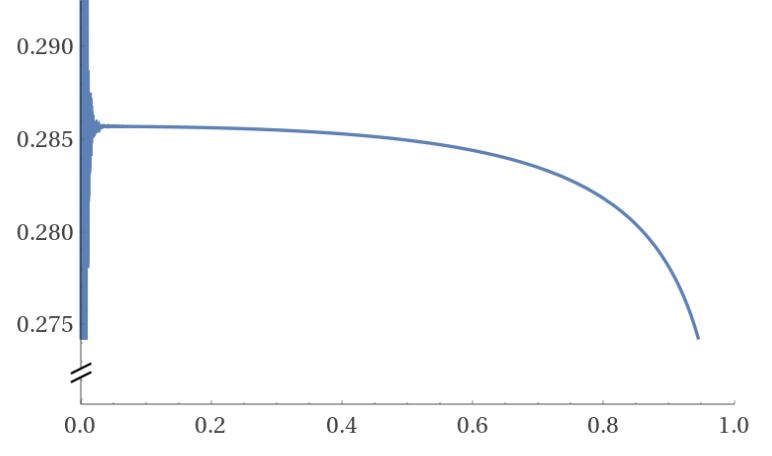}
\caption{Case $p = 3$: Plot of $y \mapsto e(y)$ for $d = 2$. }
\label{plot p=3 d=2}
\end{figure}

\begin{figure}[h]
\centering
\includegraphics[width=0.7\textwidth]{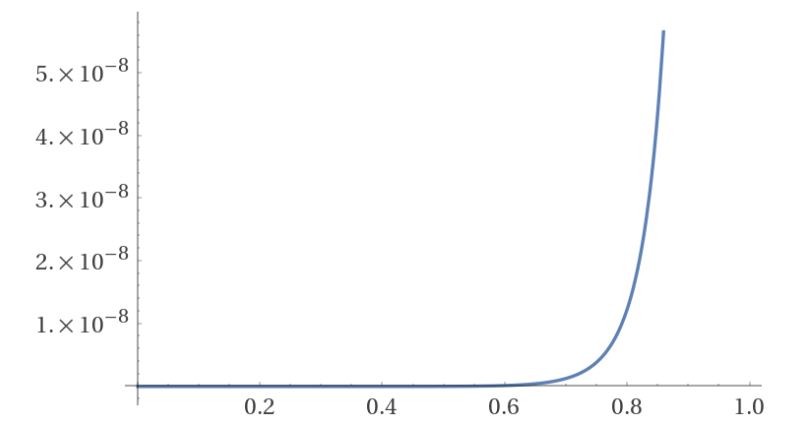}
\caption{Case $p = 3$: Plot of $y \mapsto e_1(y) - \frac{1}{5} e_2(y)$ for $d = 1$. }
\label{plot p=3 alt}
\end{figure}

The difference to higher dimensions $d \geq 2$ becomes clear when looking at the plot of $y \mapsto e(y)$ for $p=3$ and $d=2$ in Figure \ref{plot p=3 d=2}, which shows that $\widetilde{\mathcal E}(u_\beta)$ is decreasing in $\beta \in (0,1)$. This is consistent with the results of \cite{Koenig} and with the fact that $\lim_{\beta \to 1} \widetilde{\mathcal E}(u_\beta) = 1 - 2^{-\frac{2s}{d}} < \frac{4s}{d + 2s +2} = \lim_{\beta \to 0} \widetilde{\mathcal E}(u_\beta)$ for all $d \geq 2$, $s \in (0, d/2)$. (On the other hand, for $d = 1$ one has indeed $1 - 2^{-2s} > \frac{4s}{1 + 2s +2}$, consistently with Figure \ref{plot}.)

The wild oscillations at the left end of the graphs in Figures \ref{plot} and \ref{plot p=3 d=2} come with near certainty from numerical instabilities only. Their presence is greatly reduced if one plots for instance the less singular expression $e_1(y) - \frac{1}{5} e_2(y)$. Indeed, Figure \ref{plot p=3 alt}   confirms graphically that for $d = 1$ one has $e_1(y) - \frac{1}{5} e_2(y) > 0$, which is equivalent to \eqref{ineq p=3}. 

One may appreciate how close the graphs in all figures remain to their respective extremal values on a large part of the interval $(0,1)$.

\subsection{The case $p = 4$}
\label{subsection p=4}

Another easy test case which gives algebraically simple expressions is $p = 4$, i.e., $s = \frac{d}{4}$.  We again evaluate $\widetilde{\mathcal E}(u_\beta)$ for $u_\beta = v_\beta + v_{-\beta}$, with $v_\beta$ given by \eqref{v beta definition}. Since we proceed analogously to the case $p = 3$, we give fewer details here for the sake of brevity. 

Similar computations as in the proof of Lemma \ref{lemma tilde E} give 
\begin{align*}
\widetilde{\mathcal E}(u_\beta) &= \frac{1 + I(\gamma(\beta)) - 2^{-1/2} (1 + 4 I(\gamma(\beta)) + \frac{3}{|\mathbb S^d|} \int_{\mathbb S^d} v_\beta^2 v_{-\beta}^2 )^{1/2}}{1 + I(\gamma(\beta)) - 2 I(\beta)^2},
\end{align*}
where again 
\[ I(\beta) = \frac{1}{|\mathbb S^d|} \int_{\mathbb S^d} v_\beta. \]

\begin{figure}[h]
\centering
\includegraphics[width=0.7\textwidth]{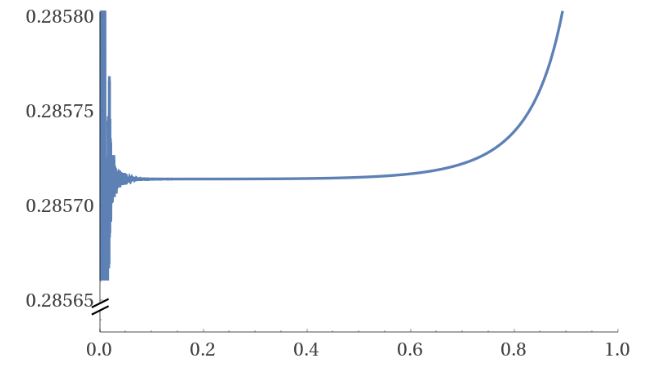}
\caption{Case $p = 4$: Plot of $y \mapsto f(y)$ for $d = 1$.}
\label{plot p=4}
\end{figure}

Using the computations and notations from the proof of Lemma \ref{lemma gamma}, by conformal invariance we have
\[ \int_{\mathbb S^d} v_\beta^2  v_{-\beta}^2  = \int_{\R^d} B_{\lambda(\beta)}^2  B_{\lambda(-\beta)}^2 = \int_{\R^d} B_{\lambda(\beta) \lambda(\beta')^{-1}} B  =  \int_{\mathbb S^d} v_{\gamma(\beta)}^2, \]
and thus 
 \begin{align*}
\frac{1}{|\mathbb S^d|} \int_{\mathbb S^d} v_\beta^2 v_{-\beta}^2 &= \frac{1}{|\mathbb S^d|} \int_{\mathbb S^d} v_{\gamma(\beta)}^2 = I_2 (\gamma(\beta)). 
\end{align*}

Thus, as in the case $p= 3$, with $y = \frac{4 \beta}{(1 + \beta)^2}$ we can write 
\[
\widetilde{\mathcal E}(u_\beta) = \frac{1 + m(y) - 2^{-1/2} (1 + 4 m(y) + 3 m_2(y) )^{1/2}}{1 + m(y) - 2  m(1 - \sqrt{1-y})^2} =: f(y),
\]
with $m$ as in \eqref{m definition} (for $p = 4$), and
\[ m_2(y) := I_2(\gamma(\beta)) = c_d (1 - y)^\frac{d}{4} {_2F_1} \left( \frac{d}{2}, \frac{d}{2} ; d ; y \right), \]
by Lemma \ref{lemma hypergeo}. 

Written in this form, we can plot the function $y \mapsto f(y)$ for $d = 1$, see Figure \ref{plot p=4}. The qualitative behavior of the plot is exactly the same as in the case $p=3$ we have studied before: $f(y)$ is strictly increasing in $y  \in (0,1)$, with $\lim_{y \to 0} f(y) = \frac{2}{7} = c_{BE}^\text{loc}\left(\frac{1}{4}\right)$.

\textbf{A sign-changing family of test functions. } When $p=4$, since $u^4 = |u|^4$, it is particularly easy to study the value of ${\mathcal E}$ for sign-changing competitors. It is natural to look at the family 
\[ w_\beta = v_\beta - v_{-\beta} \]
with $v_\beta$ again as in \eqref{v beta definition}. It is easy to see that $\mathcal E(u) \leq \mathcal E(|u|)$ for every $u$, and so it seems reasonable to expect that the $w_\beta$ are actually even stronger competitors than the $u_\beta = v_\beta + v_{-\beta}$. However, we shall see that this is not true. Actually, we will see (still numerically) that 
\begin{equation}
\label{ineq signchanging}
\widetilde{\mathcal E}(w_\beta) > 1 - 2^{-1/2} > c_{BE}^{\text{loc}}\left(\frac{1}{4}\right) = \frac{2}{7}. 
\end{equation} 
Indeed, since $\int_{\mathbb S^1} w_\beta = 0$, the infimum in $\dist(w_\beta, \mathcal C)$ is realized by the zero function and therefore simply $\dist(w_\beta, \mathcal C) = (w_\beta, P_s w_\beta)$ in this case. With practically the same computations as in the previous cases, we then obtain 
\begin{align*}
\widetilde{\mathcal E}(w_\beta) = \frac{1 + m(y) - 2^{-1/2} (1 - 4 m(y) + 3 m_2(y))}{1 + m(y)} =:g (y).
\end{align*} 

Plotting shows that $\widetilde{\mathcal E}(w_\beta)$ is strictly decreasing, see Figure \ref{plot p=4 signchanging}.

Moreover, as $\beta \to 1$ we have
\[ \dist(w_\beta, \mathcal C) = (w_\beta, P_s w_\beta)  = 2 (v_\beta, P_s v_\beta) + o(1) = 2 \dist(w_\beta, \mathcal M) + o(1). \]
Therefore,
\[ \lim_{\beta \to 1}  \widetilde{\mathcal E}(w_\beta) = \frac{1}{2} \lim_{\beta \to 1}  {\mathcal E}(w_\beta) =  \frac{1}{2} \lim_{\beta \to 1}  {\mathcal E}(u_\beta) = 1 - 2^{-1/2} > \frac{2}{7} = c_{BE}^\text{loc} \left( \frac{1}{4}\right). \]
where we used \eqref{c 2 bubble}. 

\begin{figure}[h]
\centering
\includegraphics[width=0.7\textwidth]{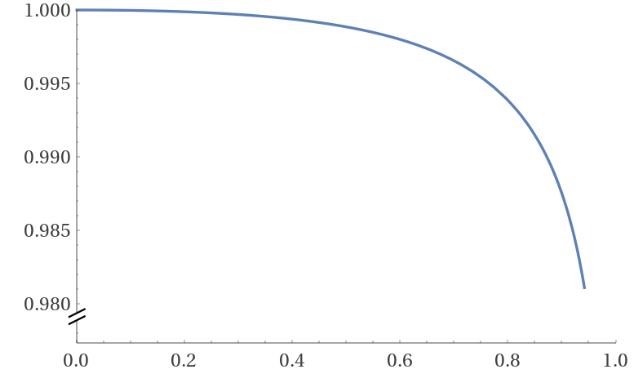}
\caption{Case $p = 4$: Plot of $y \mapsto g(y)$ for $d = 1$.} 
\label{plot p=4 signchanging}
\end{figure}

\end{document}